\documentclass[12pt]{amsart}
\usepackage{amsmath}
\usepackage{amsfonts}
\usepackage{amssymb}
\usepackage{epsfig}
\usepackage{epstopdf}
\usepackage{color}

\newtheorem{theorem}{Theorem}
  \newtheorem{lemma}[theorem]{Lemma}

   \newtheorem{corollary}[theorem]{Corollary}

  \theoremstyle{definition}

\newcommand{\RR}{\mathbb R}
\newcommand{\QQ}{\mathbb Q}

\DeclareMathOperator{\trop}{trop}
\DeclareMathOperator{\val}{val}

\date{\today}
 
 \title[]{Tropicalizing the positive semidefinite cone}
 \author{Josephine Yu}
\address{}
\email {}

\date{\today}

\begin{document}
 \maketitle
 
 \begin{abstract}
 We study the tropicalization of the cone of positive semidefinite matrices over the ordered field of real Puiseux series.  
The tropical PSD matrices form the normal cone of the Newton polytope of the symmetric determinant at the vertex corresponding to the product of diagonal entries.  We find generators and defining inequalities of the cone. The PSD tropical quadratic forms are those that induce the trivial subdivision on the standard simplex dilated by two.  We also show that the tropical PSD cone is the tropical convex hull of the set of symmetric matrices of tropical rank one and that every tropical PSD matrix can be factored as a tropical product of a matrix and its transpose.
 \end{abstract}

\medskip

 \section{Introduction}


Let $R =\bigcup_{n \geq 1} \RR((t^{1/n}))$ be the field of Puiseux series with real coefficients, where $\RR((t))$ denotes the field of formal Laurent series in $t$ with coefficients in $\RR$.  There is a non-archimedean valuation $\val : R^* = R \backslash \{0\} \rightarrow \QQ$ where $\val(a)$ is the lowest exponent appearing with a non-zero coefficient in $a$.  The field $R$ is an ordered field, with $a > b$ whenever the leading coefficient of $a-b$ is positive.  The map $\val$ naturally extends to $(R^*)^N$ coordinatewise for any positive integer $N$.  For any subset $S$ of $(R^*)^N$, let $\trop(S) = \{\val(x) : x \in S\} \subset \QQ^N$.  This notion of tropicalization was used in \cite{DevelinYu} to study tropical polytopes.

Let $\Sigma_n$ be the cone of $n \times n$ positive semidefinite (PSD) matrices over $R$, consisting of the matrices all of whose principal minors are non-negative. Let $\Sigma_n^*$ be the subset of $\Sigma_n$ consisting of the PSD matrices without any zero entries.  Our main object of interest is its tropicalization
 $$
\trop(\Sigma_n^*) = \{ \val(A) : A \in \Sigma_n^* \}.
$$
We will refer to the matrices in this set as the {\em tropical PSD matrices}.

\section{Generators and defining inequalities of $\trop(\Sigma_n^*)$}

The set $\trop(\Sigma_n^*)$ lies in the space of $n \times n$ symmetric matrices over $\QQ$.  For any $1\leq i,j \leq n$, let $e_{ij}$ be the matrix with $1$ at the $(i,j)$ and $(j,i)$ entries and $0$ everywhere else.

We will work with the tropical semiring $(\QQ, \oplus, \odot)$ where $a \oplus b = \min(a,b)$ and $a \odot b = a+b$. The {\em tropical determinant} of an $n \times n$ matrix $[a_{ij}]$ is obtained from the usual determinant by replacing addition and multiplication with $\oplus$ and $\odot$ respectively, disregarding the signs.  Thus it is the expression
$$
\min_{\sigma \in S_n}  (a_{1 \sigma(1)} + a_{2 \sigma(2)} + \dots + a_{n \sigma(n)})
$$
where $S_n$ denotes the group of permutations of $\{1,2,\dots,n\}$. 

\begin{theorem}
\label{thm:cone}
  The set $\trop(\Sigma_n^*)$ consists of all symmetric tropical matrices for which the minimum in the tropical determinant is attained at the identity permutation.  It is a closed polyhedral cone defined by inequalities
\begin{equation}
\label{eqn:ineq}
x_{ii} + x_{jj} \leq 2 x_{ij} \text{ for }i \neq j.
\end{equation}
It is generated as a cone by the rays $\{e_{ij} : i \neq j\}$ together with a lineality space spanned by $\{2 e_{ii} + \sum_{j \neq i} e_{ij} : 1 \leq i \leq n \}$. 
\end{theorem}

\begin{proof}
Let $[a_{ij}]$ be a symmetric matrix in $\trop\Sigma_n^*$. Then $[a_{ij}] = \val([\alpha_{ij}])$ for some PSD matrix $[\alpha_{ij}]$ over $R$.  For any $i \neq j$, $\alpha_{ii}\alpha_{jj} - \alpha_{ij}^2 \geq 0$, so $a_{ii} + a_{jj} = \val(\alpha_{ii}\alpha_{jj}) \leq \val(\alpha_{ij}^2) = 2 a_{ij}$.

Suppose that the minimum in the tropical determinant of $A$ is attained at a permutation $\sigma$ containing a cycle $(1,2,\dots,k)$.
Since $a_{ii}+a_{jj} \leq 2 a_{ij}$ for all $i \neq j$, we have
\begin{align*}a_{11} + a_{22}+\cdots +  a_{kk} &= \frac{1}{2} \left( (a_{11}+a_{22}) + (a_{22}+a_{33}) + \cdots + (a_{kk}+a_{11}) \right) \\
& \leq a_{12} + a_{23} + \cdots + a_{k1}.
\end{align*}
Hence replacing the cycle $(1,2,\dots,k)$  with the identity permutation in $\sigma$ gives another permutation attaining the minimum in the tropical determinant.  After removing all the cycles this way, we see that the minimum is attained at the identity permutation as well.

On the other hand, let $[a_{ij}]$ be a symmetric tropical matrix for which the minimum in the tropical determinant is attained at the identity permutation.  Then the minimum in every principal minor is attained at the identity permutation as well.  Let $[\alpha_{ij}]$ be a symmetric matrix over $R$ where $\alpha_{ij} = \pm t^{a_{ij}}$ for $i \neq j$ and $\alpha_{ii} = n! t^{a_{ii}}$.  Then for every principal minor of $[\alpha_{ij}]$, the product of diagonal entries has the lowest exponent and the coefficient is too large to be canceled out by the other terms.  Thus all the principal minors are positive, and $[\alpha_{ij}]$ is a PSD matrix whose valuation is $[a_{ij}]$.

We have shown that a matrix is tropical PSD if and only if it satisfies the inequalities (\ref{eqn:ineq}).  The matrices $2 e_{ii} + \sum_{j \neq i} e_{ij}$ attain all the inequalities at equality, so they span the lineality space of $\trop(\Sigma_n^*)$.
It is clear that each $e_{ij}$ for $i\neq j$ satisfies (\ref{eqn:ineq}).  On the other hand, for any symmetric matrix satisfying (\ref{eqn:ineq}), the off-diagonal entries can be made smaller so that all the inequalities in (\ref{eqn:ineq}) are attained at equality.  This shows that $\trop(\Sigma_n^*)$ is positively spanned by $\{e_{ij} : i\neq j\}$ and the lineality space. 
\end{proof}

In fact, it follows from the proof that $\trop(\Sigma_n^*) = \trop(\Sigma_n^* \cap C)$ where $C$ is any orthant in the space of symmetric matrices corresponding to a choice of signs for the off-diagonal entries, because we can lift a tropical PSD matrix to a matrix in $\Sigma_n^*$ with any desired sign pattern.  Note that even the matrices on the boundary of $\trop(\Sigma_n^*)$ can be lifted to matrices in the interior of $\Sigma_n^*$.

The theorem above says that $\trop(\Sigma_n^*)$ is the normal cone of the Newton polytope of the symmetric determinant at the vertex corresponding to the identity permutation.  This polytope is the projection of the Birkoff polytope onto the space of symmetric matrices.  It is linearly isomorphic to the polytope of symmetric doubly stochastic matrices.

\section{PSD quadratic forms and triangulations of $2 \Delta_{n-1}$}

The $n \times n$ symmetric matrices can be identified with quadratic forms in $n$ variables via the correspondence $A \mapsto y^T A y$ where $y$ is the column vector of $n$ variables.  We will call a tropical quadratic form PSD if it can be written as $y^T \odot A \odot y$ for some tropical PSD matrix $A$, where $\odot$ denotes the tropical matrix multiplication.

The determinant of the $n \times n$ generic symmetric matrix is the discriminant of quadratic forms.  Analogously to its classical counterpart, it follows from Theorem~\ref{thm:cone} that the tropical PSD cone $\trop(\Sigma_n^*)$ is the closure of a connected component in the complement of the tropical discriminant.  Thus the tropical PSD quadratic forms can be characterized by subdivisions they induce on the Newton polytope \cite{GKZ, DFS}.

Let $2 \Delta_{n-1}$ be the polytope consisting of points with non-negative coordinates that sum to $2$.  Any tropical quadratic polynomial corresponding to a symmetric matrix $[a_{ij}]$ induces a  subdivision on $2 \Delta_{n-1}$ as follows.  First lift the point $e_i + e_j$ to height $a_{ij}$ for each pair $i,j$, then take the convex hull.  The projection of the lower faces forms a subdivision of $2 \Delta_{n-1}$.  We call a subdivision {\em trivial} if it consists of a single maximal cell.
\newpage
\begin{theorem}
\label{thm:subdivision}
The tropical PSD quadratic forms are those that induce the trivial subdivision on $2 \Delta_{n-1}$.
\end{theorem}

\begin{proof}
For a matrix $[a_{ij}] \in \trop(\Sigma_n^*)$, we have $a_{ii} + a_{jj} \leq 2 a_{ij}$ for any $i \neq j$, so the point $e_i+e_j \in 2 \Delta_{n-1}$ is lifted at least as high as the midpoint of the line segment between lifts of $2 e_i$ and $2 e_j$.  It follows that the only lower facet is the convex hull of the lifts of $2 e_i$ for $i=1,\dots,n$.
\end{proof}

\section{Tropical generators and symmetric Barvinok rank}

A tropical matrix $A = [a_{ij}]$ has tropical rank one if $a_{ij}+a_{kl} = a_{il}+a_{kj}$ for all $i \neq k$ and $j\neq l$; in other words, the minimum is attained twice in all $2 \times 2$ tropical minors of $A$.  A matrix has tropical rank one if and only if it can be written as $u \odot v^T$ for some column vectors $u$ and $v$, i.e. it has Barvinok rank one \cite{DSS}.  

\begin{lemma}
\label{lem:rank1}
The space of $n \times n$ symmetric matrices of tropical rank one is a linear space spanned by $\{2 e_{ii} + \sum_{j \neq i} e_{ij} : i = 1,2,\dots,n\}$.
\end{lemma}

\begin{proof}
All matrices in the linear space is obtained from the all-zero matrix, which has tropical rank one, by tropically scaling rows and columns simultaneously, and tropical scaling does not change the tropical rank. On the other hand, all symmetric tropical rank one matrices satisfy the equations $x_{ii}+x_{jj} = 2 x_{ij}$ for $i \neq j$ that cut out the linear space under consideration.  
\end{proof}

The tropical convex hull of a set consists of all finite tropical linear combinations of the set.  It is the analogue of positive hull in classical convexity theory.

\begin{theorem}
\label{thm:rank1}
The set $\trop(\Sigma_n^*)$ is the tropical convex hull of the set of symmetric matrices of tropical rank one.  
\end{theorem}

For the proof, we will use the ideas used by Develin to study tropical secant sets of linear spaces \cite{DevelinTropSecant}.

\begin{proof} 
A symmetric matrix $A = [a_{ij}]$ defines a function on the lattice points $2 \Delta_{n-1}$ by $e_i + e_j \mapsto a_{ij}$.  By Lemma~\ref{lem:rank1}, the tropical rank one symmetric matrices are precisely those that induce affine linear functions on $2 \Delta_{n-1}$.  The tropical linear combinations of the tropical rank one matrices give functions on $2 \Delta_{n-1}$ that arise as the minimum of some affine linear functions.  In other words, if we lift each point $e_i + e_j$ to height $a_{ij}$ and take the convex hull, all the lifted points must be contained in the upper faces.  This means that the lower hull gives the trivial subdivision of $2 \Delta_{n-1}$, and the result follows from Theorem~\ref{thm:subdivision}.
\end{proof}

The {\em symmetric Barvinok rank} of a symmetric tropical matrix $A$ is the minimum integer $r$ such that $A$ can be written as the tropical sum of $r$ tropical rank one symmetric matrices.  The symmetric Barvinok rank may be infinite \cite{CartwrightChan}.  The next results follows immediately from Theorem~\ref{thm:rank1}.

\begin{corollary}
A symmetric matrix is tropical PSD if and only if its symmetric Barvinok rank is finite.
\end{corollary}

An $n \times n$ symmetric matrix is PSD if and only if it is the Gram matrix of $n$ vectors, i.e.\ the matrix of inner products among the vectors.  The analogous statement is true for tropical PSD matrices.

\begin{corollary}
A symmetric matrix is tropical PSD if and only if it can be factored as $B \odot B^T$ for some matrix $B$.
\end{corollary}

\begin{proof}
By Theorem~\ref{thm:rank1}, a symmetric matrix $A$ is tropical PSD if and only if $A = c_1 \odot (v_1 \odot v_1^T) \oplus \cdots \oplus c_r \odot (v_r \odot v_r^T)$ for some scalars $c_1,\dots,c_r$ and column vectors $v_1,\dots,v_r$.  This is equivalent to the existance of a matrix $B$ whose columns are $c_1/2 \odot v_1, \dots, c_r/2 \odot v_r$.  
\end{proof}

It may be necessary for the matrix $B$ to have more than $n$ columns, i.e.\ the symmetric Barvinok rank may be a finite number greater than $n$, but it is bounded above by $\max(n,\lfloor n^2/4 \rfloor)$  \cite{CartwrightChan}.

To compute such a decomposition of a tropical PSD matrix $A$, consider the convex hull of lattice points in $2 \Delta_{n-1}$ lifted according to $A$.  Choose a collection of upper facets that contain all the lifted lattice points in their union.  Each facet defines an affine linear function on $2 \Delta_{n-1}$, giving a tropical rank one symmetric matrix.  Their tropical sum is $A$.

\bibliographystyle{amsalpha}
\bibliography{mybib}

 \end{document}